\documentclass[12pt,a4paper]{article}  
\usepackage[a4paper]{geometry}
\usepackage[english]{babel}
\usepackage[utf8]{inputenc}
\usepackage[T1]{fontenc}
\usepackage{times}
\usepackage{amsmath}
\usepackage{amsfonts}
\usepackage{amssymb}
\usepackage{mathrsfs}
\usepackage{amsthm}
\usepackage{url}
\usepackage{graphicx}
\usepackage[usenames,dvipsnames]{xcolor}
\usepackage{tikz}
\usetikzlibrary{matrix,arrows,decorations.markings}
\theoremstyle{definition}
\newtheorem{comcnt}{Anything}[section]
\newcommand\thingy{%
\refstepcounter{comcnt}\medbreak\noindent\textparagraph\textbf{\thecomcnt.} }
\newtheorem{prop}[comcnt]{Proposition}
\newtheorem{lem}[comcnt]{Lemma}
\newtheorem{dfn}[comcnt]{Definition}

\newtheorem{cor}[comcnt]{Corollary}

\newtheorem{rmk}[comcnt]{Remark}
\newcommand{\Frac}{\mathop{\mathrm{Frac}}\nolimits}
\newcommand{\arctanh}{\mathop{\mathrm{arctanh}}\nolimits}
\mathchardef\emdash="07C\relax
\mathchardef\hyphen="02D\relax
\DeclareUnicodeCharacter{00A0}{~}
\title{The Hadwiger-Nelson problem over certain fields}
\author{David A. Madore}
\begin{document}
\maketitle

\begin{abstract}
We compute the Hadwiger-Nelson numbers $\chi(E^2)$ for certain number
fields $E$, that is, the smallest number of colors required to color
the points in the plane with coordinates in~$E$ so that no two points
at distance $1$ from one another have the same color.  Specifically,
we show that $\chi(\mathbb{Q}(\sqrt{2})^2) = 2$, that
$\chi(\mathbb{Q}(\sqrt{3})^2) = 3$, that $\chi(\mathbb{Q}(\sqrt{7})^2)
= 3$ despite the fact that the graph $\Gamma(\mathbb{Q}(\sqrt{7})^2)$
is triangle-free, and that $4 \leq \chi(\mathbb{Q}(\sqrt{3},
\sqrt{11})^2) \leq 5$.  We also discuss some results over other
fields, for other quadratic fields.  We conclude with some comments on
the use of the axiom of choice.
\end{abstract}

\section{Introduction}

\thingy The \emph{Hadwiger-Nelson problem} asks what is the minimum
number $\chi(\mathbb{R}^2)$ of colors required to color the plane
$\mathbb{R}^2$ in such a way that two points $x,x'$ at distance $1$
from one another (i.e., such that $(x'_1-x_1)^2 + (x'_2-x_2)^2 = 1$)
never have the same color.  (We refer to \cite{Soifer2009} for more
about this problem, especially to chapter 2 for basic information and
chapter 3 for a historical account of how the problem emerged and how
it camed to be associated with the names of Hugo Hadwiger and Edward
Nelson.)  In other words, the question is that of the chromatic number
$\chi(\mathbb{R}^2) = \chi(\Gamma(\mathbb{R}^2))$ of the graph
$\Gamma(\mathbb{R}^2)$ whose vertices are the points of $\mathbb{R}^2$
with an edge connecting any two points at distance $1$ from one
another.

\thingy The best bounds currently known are
$4\leq\chi(\mathbb{R}^2)\leq 7$, and are proved using completely
elementary methods: the lower bound $\chi(\mathbb{R}^2) \geq 4$ is
obtained by embedding an explicit finite graph with chromatic
number $4$ in $\Gamma(\mathbb{R}^2)$ (typically \emph{Moser's
  spindle}, cf. \cite[fig. 2.2]{Soifer2009} and
\ref{embedding-mosers-spindle} below; or the \emph{Golomb graph},
cf. \cite[fig. 2.8]{Soifer2009}), whereas the upper bound
$\chi(\mathbb{R}^2) \leq 7$ is obtained by an explicit coloring
(typically by tiling the plane with hexagons of diameter just less
than $1$ and periodically coloring them using the $7$-coloring
sometimes known as ``Heawood's map'': see \cite[fig. 2.5]{Soifer2009}
for details, and cf. also \cite{Sevennec2013} for an algebraic
presentation of Heawood's map and some if its other remarkable
properties).

\thingy\label{reminder-de-bruijn-erdos} It is also worth recalling the
De Bruijn-Erdős theorem (\cite[theorem 1]{DeBruijnErdos1951}), which
guarantees that an infinite graph $G$ is $n$-colorable iff every
finite subgraphs of $G$ is (i.e., $\chi(G)$ is the upper bound of the
$\chi(G_0)$ for all finite subgraphs $G_0$ of $G$).  This relies on
some form of the axiom of choice (which we assume throughout, but see
section \ref{section-choice} for comments) and can be seen, for
example, as an immediate consequence of the compactness theorem for
propositional calculus (see, e.g., \cite[theorem 4.5]{Poizat2000})
applied to the (infinite) set of propositional variables ``vertex $v$
has color $i$'' and the (infinite) set of axioms stating that each
vertex has exactly one color and no two adjacent vertices have the
same color.

In the case of the Hadwiger-Nelson problem, this tells us that any
lower bound on $\chi(\mathbb{R}^2)$ can be obtained by finding a
finite unit-distance graph with that chromatic number (one immediate
consequence of this is that lower bounds on $\chi(\mathbb{R}^2)$ are
necessarily provable: see \ref{remarks-real-closed-fields} for
details).

\bigbreak

\thingy There are several ways the Hadwiger-Nelson problem can be
generalized.  An obvious one consists of changing the dimension from
$2$ to $d$: we refer to \cite[chapter 10]{Soifer2009} for a discussion
on the bounds known for $\chi(\mathbb{R}^d)$.  This paper is mostly
concerned with the case $d=2$, although we will keep $d$ as a variable
whenever it is irrelevant.

\thingy Another way to generalize the Hadwiger-Nelson problem is to
restrict oneself to coloring the points whose coordinates lie in a
certain subfield $E$ of $\mathbb{R}$, i.e., ask for the chromatic
number $\chi(E^2)$, or more generally $\chi(E^d)$, of the graph
$\Gamma(E^d)$ whose whose vertices are the points of $E^d$ with an
edge connecting any two points $(x_1,\ldots,x_d)$ and
$(x'_1,\ldots,x'_d)$ whenever $(x'_1-x_1)^2 + \cdots + (x'_d-x_d)^2 =
1$.

Note that even though the notion of ``distance'' might no longer be
applicable, this graph $\Gamma(E^d)$ as we have defined it, and
consequently its chromatic number $\chi(E^d)$, make sense for
\emph{any} field whatsoever, or in fact, any commutative ring (or
indeed, any kind of ring), not necessarily embeddable in $\mathbb{R}$.
So it makes sense, for example, to ask for the value of
$\chi(\mathbb{C}^2)$, which may or may not be finite (this does not
seem to have been studied, and the present author does not know
anything beyond the trivial bound $4 \leq \chi(\mathbb{R}^2) \leq
\chi(\mathbb{C}^2) \leq \infty$; but see
\ref{comparison-of-characteristics} and \ref{lorentzian-case}), or
$\chi((\mathbb{F}_p)^2)$ where $\mathbb{F}_p = \mathbb{Z}/p\mathbb{Z}$
is the finite field with $p$ elements (since
$\Gamma((\mathbb{F}_p)^2)$ is a finite graph, this can be computed for
any given $p$), or again $\chi((\mathbb{Q}_p)^2)$ where $\mathbb{Q}_p$
is the field of $p$-adic numbers (as we explain in
\ref{main-result-valuation-rings} and \ref{remark-p-adic-fields}, we
have $\chi((\mathbb{Q}_p)^2) \leq \chi((\mathbb{F}_p)^2)$ if $p\equiv
3\pmod{4}$).

\thingy One classical result (\cite[theorem 1]{Woodall1973},
cf. \cite[11.2]{Soifer2009}) is that $\chi(\mathbb{Q}^2) = 2$: this is
proved by reducing modulo $2$ (see \ref{remark-reduction-mod-4} for a
proof of this result in the formalism of this paper).  The values
$\chi(\mathbb{Q}^3) = 2$ and $\chi(\mathbb{Q}^4) = 4$ are also known
(\cite{BendaPerles2000}, cf. \cite[11.3 \& 11.4]{Soifer2009}).  For
other kinds of fields, nothing seems to have been said: the question
of finding $\chi(E^2)$ for any number field $E$, and
$\mathbb{Q}(\sqrt{2})$ in particular, is listed as an open problem
in \cite[11.6]{Soifer2009}.

The main results of this paper are that $\chi(\mathbb{Q}(\sqrt{2})^2)
= 2$ (prop. \ref{chromatic-number-of-q-sqrt-2}), that
$\chi(\mathbb{Q}(\sqrt{3})^2) = 3$
(prop. \ref{chromatic-number-of-q-sqrt-3}), that
$\chi(\mathbb{Q}(\sqrt{7})^2) = 3$ despite the fact that the graph
$\Gamma(\mathbb{Q}(\sqrt{7})^2)$ is triangle-free
(prop. \ref{chromatic-number-of-q-sqrt-7}), and that $4 \leq
\chi(\mathbb{Q}(\sqrt{3}, \sqrt{11})^2) \leq 5$
(prop. \ref{chromatic-number-of-q-sqrt-3-and-11}).  We also discuss
some results over other fields.

Yet another generalization, consisting of changing the quadratic form
$x_1^2 + \cdots + x_d^2$ used to define the distance, will be
discussed biefly in section \ref{section-quadratic-form}.  We conclude
with comments on the use of the axiom of choice in
section \ref{section-choice}.

\bigbreak

\thingy\label{conventions}\textbf{Conventions.} By a ``graph'', we
mean a set $X$ of ``vertices'', together with a set of two-element
subsets of $X$ called ``edges'', i.e., an undirected graph without
multiple edges or self-edges.  Two vertices connected by an edge are
also said to be ``adjacent''.  A \emph{graph homomorphism} $\psi\colon
G \to G'$ between graphs $G$ and $G'$ is a map $\psi$ from the set of
vertices of $G$ to that of $G'$ such that if $x$ and $x'$ are adjacent
then $\psi(x)$ and $\psi(x')$ are adjacent.

A \emph{coloring} of a graph $G$ with $n$ colors is a graph
homomorphism from $G$ to the complete (=clique) graph $K_n$ consisting
of $n$ vertices with all $n(n-1)/2$ possible edges (we also say that
$G$ is colorable using $n$ colors, or simply $n$-colorable): the set
of vertices given each color is, of course, those which map to a given
vertex of the target.  This definition makes it clear that if
$\psi\colon G \to G'$ is a graph homomorphism and $G'$ is
$n$-colorable, then so is $G$ (we are not assuming $\psi$ to be
injective): if $G' \to K_n$ is a coloring then the composite with
$\psi$ gives a coloring $G \to K_n$ which we say is obtained by
``pulling back'' the coloring of $G'$ by $\psi$.  We write $\chi(G)$
and call \emph{chromatic number} of $G$ the smallest natural number
$n$ such that $G$ is $n$-colorable, or $\infty$ if such $n$ does not
exist: by what has just been said, if $\psi\colon G \to G'$ is a graph
homomorphism then $\chi(G) \leq \chi(G')$.

Also, we stated the De Bruijn-Erdős theorem
in \ref{reminder-de-bruijn-erdos} above by considering all finite
subgraphs $G_0$ of a graph $G$ (i.e., injective graph homomorphisms
$G_0 \to G$ with $G_0$ a finite graph), but a moment's thought
suffices to see that the statement is equally valid for induced finite
subgraphs (i.e., injective graph homomorphisms $G_0 \to G$, with $G_0$
a finite graph, such that $x,x'$ are adjacent \emph{iff}
$\psi(x),\psi(x')$ are) or simply all homomorphisms $G_0 \to G$, with
$G_0$ a finite graph.  For consistency's sake, we have tried to always
use and speak of homomorphisms (even though they are often, in fact,
injective, or even embeddings of an induced subgraph, and it generally
does not matter how they are considered).

\section{Generalities and lower bounds}\label{section-generalities}

We formalize the notion suggested in the introduction:

\begin{dfn}\label{definition-main-graph}
Let $E$ be any field, or even any (nonzero\footnote{If $E$ were the
  zero ring (i.e., the ring in which $0=1$), then $\Gamma(E^d)$ would
  consist of a single vertex connected to itself by an edge, and which
  is therefore not colorable using any number of colors: we exclude
  this degenerate case because we consider only graphs with no
  self-edges.}) commutative ring, and $d \geq 1$.  We define a graph
$\Gamma(E^d)$ as follows: vertices of $\Gamma(E^d)$ are $d$-tuples
from $E$, with an edge between $(x_1,\ldots,x_d)$ and
$(x'_1,\ldots,x'_d)$ whenever $(x'_1-x_1)^2+\cdots+(x'_d-x_d)^2 = 1$.
We write $\chi(E^d) = \chi(\Gamma(E^d))$ for the chromatic number of
this graph $\Gamma(E^d)$ (possibly $+\infty$).
\end{dfn}

\thingy\label{observation-morphisms} A trivial observation: given any
morphism of (nonzero) commutative rings $\psi\colon E\to E'$, we get a
homomorphism of graphs $\Gamma(E^d) \to \Gamma(E^{\prime d})$ by
taking $x = (x_1,\ldots,x_d)$ to $\psi(x) =
(\psi(x_1),\ldots,\psi(x_d))$, where a ``homomorphism of graphs'' was
defined in \ref{conventions} (if $x,x'$ are adjacent in $\Gamma(E^d)$
then $\psi(x),\psi(x')$ are adjacent in $\Gamma(E^{\prime d})$); as
noted there, pulling back by $\psi$ any coloring of $\Gamma(E^{\prime
  d})$ gives a coloring of $\Gamma(E^d)$ with the same number of
colors, so $\chi(E^d) \leq \chi(E^{\prime d})$.  This applies in
particular to an extension of fields: if $E\subseteq E'$ are fields
then $\chi(E^d) \leq \chi(E^{\prime d})$ (something which was obvious
from the start); it also applies to a quotient ring: if $A$ is a
commutative ring with an ideal $I$, then $\chi(A^d) \leq
\chi((A/I)^d)$.

\bigbreak

To get a lower bound on $\chi(E^2)$, as we recalled
in \ref{reminder-de-bruijn-erdos}, we need to construct a homomorphism
from a finite graph to $\Gamma(E^2)$.  Practically the only two useful
graphs which are known in this context are the triangle and Moser's
spindle, which we now discuss:

\begin{lem}\label{embedding-triangle}
If $E$ is a field of characteristic $\neq 2$ in which $3$ is a square
(and we write $\sqrt{3}$ for a square root of it), then the points
$(0,0)$, $(1,0)$ and $(\frac{1}{2},\frac{\sqrt{3}}{2})$ of $E^2$
induce an homomorphism from the triangle graph $C_3$ to $\Gamma(E^2)$,
showing that $\chi(E^2) \geq 3$.
\end{lem}

The proof is contained in the statement (together with the obvious
fact that $\chi(C_3) = 3$).

We recall that, by quadratic reciprocity, $3$ is a square in
$\mathbb{F}_q$ iff $q\equiv \pm 1\pmod{12}$ or $q$ is a power of $3$
(or of $2$).

\begin{lem}\label{embedding-mosers-spindle}
If $E$ is a field of characteristic $\neq 2$ in which $3$ and $11$ are
squares, then \emph{Moser's spindle} graph (displayed below) admits a
homomorphism to $\Gamma(E^2)$, showing that $\chi(E^2) \geq 4$.
\end{lem}

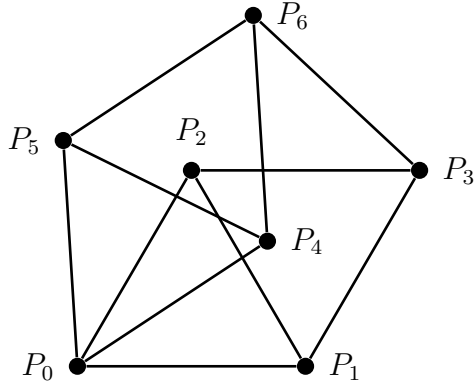
\begin{center}
\begin{tikzpicture}
\begin{scope}[scale=3]
\begin{scope}[every node/.style={circle,fill,inner sep=0.8mm}]
\node [label={left:$P_0$}] (v0) at (0,0) {};
\node [label={right:$P_1$}] (v1) at (1.000,0) {};
\node [label={above:$P_2$}] (v2) at (0.500,0.866) {};
\node [label={right:$P_3$}] (v3) at (1.500,0.866) {};
\node [label={right:$P_4$}] (v4) at (0.833,0.553) {};
\node [label={left:$P_5$}] (v5) at (-0.062,0.998) {};
\node [label={right:$P_6$}] (v6) at (0.771,1.551) {};
\end{scope}
\begin{scope}[line width=1pt]
\draw (v0) -- (v1);
\draw (v0) -- (v2);
\draw (v0) -- (v4);
\draw (v0) -- (v5);
\draw (v1) -- (v2);
\draw (v1) -- (v3);
\draw (v2) -- (v3);
\draw (v3) -- (v6);
\draw (v4) -- (v5);
\draw (v4) -- (v6);
\draw (v5) -- (v6);
\end{scope}
\end{scope}
\end{tikzpicture}
\end{center}

\begin{proof}
Let $\sqrt{3},\sqrt{11}$ be square roots of $3$, $11$ in $E$, and
$\sqrt{33}$ their product.  Then a straightforward computation shows
that the points
\[
\begin{array}{c}
P_0 = (0,0),\; P_1 = (1,0),\; P_2 = (\frac{1}{2},\frac{\sqrt{3}}{2}),\;
P_3 = (\frac{3}{2},\frac{\sqrt{3}}{2}),\\
P_4 = (\frac{5}{6},\frac{\sqrt{11}}{6}),\;
P_5 = (\frac{5}{12}-\frac{\sqrt{33}}{12}, \frac{5\sqrt{3}}{12}+\frac{\sqrt{11}}{12}),\;
P_6 = (\frac{5}{4}-\frac{\sqrt{33}}{12}, \frac{5\sqrt{3}}{12}+\frac{\sqrt{11}}{4})\;
\end{array}
\]
form a graph with edges $\{P_0,P_1\}$, $\{P_0,P_2\}$, $\{P_1,P_2\}$,
$\{P_1,P_3\}$, $\{P_2,P_3\}$, $\{P_0,P_4\}$, $\{P_0,P_5\}$,
$\{P_4,P_5\}$, $\{P_4,P_6\}$, $\{P_5,P_6\}$, $\{P_3,P_6\}$,
represented above.  It is clear that this graph has chromatic
number $4$.
\end{proof}

We recall that, by quadratic reciprocity, $11$ is a square in
$\mathbb{F}_q$ iff $q$ is congruent modulo $44$ to an element of
$\{\pm 1, \pm 9, \pm 5, \pm 7, \pm 19\}$ or $q$ is a power of $11$ (or
of $2$).

\bigbreak

\begin{rmk}\label{fields-of-charac-2}
Let us get the fields of characteristic $2$ out of the way with the
following remark.

If $E$ is a field of characteristic $2$, then
$(x'_1-x_1)^2+\cdots+(x'_d-x_d)^2 = 1$ is equivalent to
$(x'_1-x_1)+\cdots+(x'_d-x_d) = 1$, that is, $\lambda(x'-x) = 1$ where
$\lambda$ is the $E$-linear form $(z_1,\ldots,z_d) \mapsto
z_1+\cdots+z_d$.  Complete $1$ to a basis of $E$ as an
$\mathbb{F}_2$-vector space (this uses the axiom of choice) and let
$\hat\lambda(z)$ be the coordinate on $1$ of $\lambda(z)$: then we get
a coloring of $E^d$ with two colors if we choose the color of $x$
according to the value of $\hat\lambda(x) \in \mathbb{F}_2$.  Since
obviously $1$ color does not suffice (for $d\geq 1$), this shows that
$\chi(E^d) = 2$.
\end{rmk}

\section{Obtaining upper bounds by reduction}\label{section-upper-bounds}

\thingy \textbf{Informal discussion} (only used to motivate what
follows).  Assume $K$ is a number field (i.e., a finite extension
of $\mathbb{Q}$), and $\mathfrak{p}$ is a maximal ideal of the ring of
integers $\mathcal{O}_K$ of $K$.  (More generally, the more
algebraically oriented reader might wish to assume that $\mathfrak{p}$
is a maximal ideal of a Dedekind domain $\mathcal{O}_K$ with fraction
field $K$.)  We call $\kappa = \mathcal{O}_K/\mathfrak{p}$ the residue
field of $\mathfrak{p}$.

(An example of such a situation occurs when $K = \mathbb{Q}$ so that
$\mathcal{O}_K = \mathbb{Z}$ and $\mathfrak{p}$ is the ideal generated
by an ordinary prime number $p\geq 2$, with $\kappa = \mathbb{F}_p$;
but this will not give any useful consequences since the value of
$\chi(\mathbb{Q}^2)$ is known.  We refer, e.g., to
\cite[I.§1--3]{Neukirch1999} for background on number fields.)

We would like to obtain an upper bound on $\chi(K^d)$ by comparing it
with $\chi(\kappa^d)$ using some kind of ``reduction
mod $\mathfrak{p}$'' argument and
applying \ref{observation-morphisms}.  Let us informally discuss how
this can be done.

One cannot reduce the elements of $K$ mod $\mathfrak{p}$ as one can
for the elements of $\mathcal{O}_K$, so there is no obvious map from
(affine space) $K^d$ to $\kappa^d$.  However, as we now explain, there
\emph{is} a natural ``reduction'' map $\mathbb{P}^d(K) \to
\mathbb{P}^d(\kappa)$, where $\mathbb{P}^d$ refers to projective
$d$-dimensional space.  (Recall that, for $K$ an arbitrary field,
$\mathbb{P}^d(K)$ is the set of $(z_0,\ldots,z_d) \in K^{d+1}$, not
all zero, modulo multiplication by a nonzero constant; we write
$(z_0:\cdots:z_d)$ for the class of $(z_0,\ldots,z_d)$ in
$\mathbb{P}^d(K)$, and which we consider affine space $K^d$ embedded
in $\mathbb{P}^d(K)$ by $(z_1,\ldots,z_d) \mapsto
(1:z_1:\cdots:z_d)$).  This reduction map $\mathbb{P}^d(K) \to
\mathbb{P}^d(\kappa)$ can be defined by ``clearing denominators'',
i.e. multiplying the homogeneous coordinates $z_0,\ldots,z_d$ by an
appropriate element of $K$ so that they all lie in $\mathcal{O}_K$
with at least one of them not in $\mathfrak{p}$ (i.e., their valuation
with respect to $\mathfrak{p}$ are all nonnegative and not all
positive), and then reducing mod $\mathfrak{p}$ to get an element of
$\mathbb{P}^d(\kappa)$.  In more sophisticated terms, this works
because $\mathbb{P}^d(K)$ actually coincides with
$\mathbb{P}^d(\mathcal{O}_K)$ when $\mathcal{O}_K$ is a Dedekind ring
with quotient field $K$ (the projective space $\mathbb{P}^d(A)$ over a
ring $A$ is more delicate to define than over a field: for example, it
is the set of projective submodules $L$ of rank $1$ of $A^{d+1}$,
where a projective submodule of $A^{d+1}$ means there is $M$ such that
$L \oplus M = A^{d+1}$, and rank $1$ means that for any quotient of
$A$ by a maximal ideal, the corresponding quotient of $L$ is
$1$-dimensional vector space).

So now we have a reduction map $K^d \to \mathbb{P}^d(K) \to
\mathbb{P}^d(\kappa)$ and we assume we find a coloring of
$\Gamma(\kappa^d)$ with few colors: $\kappa^d$ is embedded in
$\mathbb{P}^d(\kappa)$ but in general this does not avail us because
many points of $K^d$ will reduce to the ``hyperplane at infinity''
$\{z_0=0\}$ in $\mathbb{P}^d(\kappa)$ (the complement of the subset
$\kappa^d \subseteq \mathbb{P}^d(\kappa)$).  However, if it so happens
that $z_1^2 + \cdots + z_d^2$ is ``anisotropic'', meaning that $z_1^2
+ \cdots + z_d^2 = 0$ has no non-trivial solution in $\kappa$ (we also
say that the quadric $\{z_1^2 + \cdots + z_d^2 = z_0^2\} \subseteq
\mathbb{P}^d(\kappa)$ has no points at infinity, this quadric being
the projective completion of the affine ``unit circle'' $\{z_1^2 +
\cdots + z_d^2 = 1\} \subseteq \kappa^d$), then by translating we can
``stay away from infinity'', and we can get a coloring of $K^d$ from
one of $\kappa^d$, as explained by the following proposition and
corollary:

\begin{prop}\label{main-result-valuation-rings}
Let $A$ be a valuation ring with valuation $v$: write $\mathfrak{m} :=
\{x\in A : v(x)>0\}$ for its (unique) maximal ideal, $\kappa :=
A/\mathfrak{m}$ the residue field, and $K := \Frac(A)$ for the field
of fractions of $A$.  Assume that the quadratic form $z_1^2 + \cdots +
z_d^2$ is \emph{anisotropic} over $\kappa$ (that is, $z_1^2 + \cdots +
z_d^2 = 0$ has no solution in $\kappa$ other than the trivial
$(z_1,\ldots,z_d) = (0,\ldots,0)$).

Then there is a graph homomorphism $\psi\colon \Gamma(K^d) \to
\Gamma(A^d)$.  Recalling (\ref{observation-morphisms}) that there are
also obvious graph homomorphisms $\Gamma(A^d) \to \Gamma(K^d)$ and
$\Gamma(A^d) \to \Gamma(\kappa^d)$, we have: $\chi(K^d) = \chi(A^d)
\leq \chi(\kappa^d)$.
\end{prop}

\thingy For the reader's convenience, we recall, cf. e.g.,
\cite[§10]{Matsumura1989} or \cite[II.§3]{Neukirch1999}, that a
\emph{valuation ring} $A$ is an integral domain such that every
element $x$ of its field of fractions $K$ satisfies $x \in A$ or
$x^{-1} \in A$.  Such a ring has a unique maximal ideal.  The
valuation $v$ can be defined as the quotient map from $K^\times$ to
the abelian group $K^\times/A^\times$ or ``value group'', where
$K^\times = K\setminus\{0\}$ is the group of nonzero elements of $K$,
and $A^\times$ is the group of invertible elements of $A$; this is
ordered by $v(x) > v(y)$ iff $x/y \in A$; it is extended by putting
$v(0)=\infty$, a symbol greater than all others (in particular, $A =
\{x\in K : v(x)\geq 0\}$).  The essential properties of a valuation
are: (o) $v(x) = \infty$ iff $x=0$, (i) $v(xy) = v(x) + v(y)$ and
(ii) $v(x+y) \geq \min(v(x), v(y))$ (and it follows that, in the
latter, equality in fact holds if $v(x) \neq v(y)$).

In the applications to integers of number fields, or Dedekind domains
in general, $A$ will be a \emph{discrete} valuation ring, meaning that
the value group $K^\times/A^\times$ is simply $\mathbb{Z}$ with the
usual order (although sometimes it will be more convenient to
normalize it differently, e.g., $\frac{1}{2}\mathbb{Z}$).  Nothing
will be lost if the reader assumes this from the start.

\begin{proof}[Proof of \ref{main-result-valuation-rings}]
First observe the following fact: if $x_1^2 + \cdots + x_d^2 = 1$ with
$x_i\in K$, then in fact $x_i \in A$ for all $i$.  Indeed, assume on
the contrary that $x_1^2 + \cdots + x_d^2 = 1$ and $v(x_i)<0$ for
some $i$; and let $u$ be such that $v(x_u)$ is the smallest (=most
negative).  Then each $z_i := x_i/x_u$ belongs to $A$, and $1/x_u$
belongs to $\mathfrak{m}$ (since it has positive valuation), and
$(x_1/x_u)^2 + \cdots + (x_d/x_u)^2 = 1/x_u^2$.  Reducing this
equation mod $\mathfrak{m}$ gives $\bar z_1^2 + \cdots + \bar z_d^2 =
\bar 0$ in $\kappa$.  By assumption, all $\bar z_i$ must be $0$, but
this contradicts $z_u = 1$.

Now define an equivalence relation $\approx$ on $K^d$ by
$(x_1,\ldots,x_d) \approx (x'_1,\ldots,x'_d)$ iff $x'_i - x_i \in A$
for every $i$.  The fact noted in the previous paragraph means that
each edge of $\Gamma(K^d)$ connects two points in the same equivalence
class for $\approx$.  For each $\approx$-equivalence class $C$ of
$\Gamma(K^d)$ (i.e., each element of the quotient group $K^d/A^d$),
choose a representative $\xi_C \in C$, and define $\psi$ on $C$ as
taking $x \in C$ to $x - \xi_C$, which by definition of $\approx$
belongs to $A^d$.  Clearly if $x,x'$ are adjacent in $\Gamma(K^d)$,
they are in the same equivalence class $C$ for $\approx$, and $x -
\xi_C$ and $x' - \xi_C$ are also adjacent, so that $\psi(x)$ and
$\psi(x')$ are.  So we have defined a graph homomorphism $\psi \colon
\Gamma(K^d) \to \Gamma(A^d)$.
\end{proof}

\begin{cor}\label{main-result-reduction-number-fields}
Assume $K$ is a number field, $\mathcal{O}_K$ its ring of integers,
and $\mathfrak{p}$ a maximal ideal of $\mathcal{O}_K$ such that the
cardinality $q =: \mathrm{N}(\mathfrak{p})$ of the residue field
$\mathcal{O}_K/\mathfrak{p} = \mathbb{F}_q$ is congruent to $3$
mod $4$.  Then there is a graph homomorphism $\Gamma(K^2) \to
\Gamma((\mathbb{F}_q)^2)$.  In particular, $\chi(K^2) \leq
\chi((\mathbb{F}_q)^2)$.
\end{cor}
\begin{proof}
If $q \equiv 3\pmod{4}$ then $-1$ is not a square in $\mathbb{F}_q$:
then the quadratic form $z_1^2 + z_2^2 = 0$ is anisotropic (for if
there were a solution with, say, $z_1 \neq 0$, we would have
$(z_2/z_1)^2 = -1$).

Now let $A = \mathcal{O}_{K,\mathfrak{p}}$ be the localization of
$\mathcal{O}_K$ at $\mathfrak{p}$, i.e., the subring of $K$ consisting
of quotients of elements of $\mathcal{O}_K$ whose denominator is not
in $\mathfrak{p}$: since $\mathcal{O}_K$ is a Dedekind domain
(\cite[prop. I.12.8]{Neukirch1999}), this localization is a discrete
valuation ring (\cite[prop. I.11.5]{Neukirch1999}) with fraction
field $K$ and residue field $\mathbb{F}_q$.
Proposition \ref{main-result-valuation-rings} gives the conclusion.
\end{proof}

\begin{rmk}
We stated the corollary for $d=2$.  One does not obtain anything
interesting for $d>2$ because a quadratic form in $\geq 3$ variables
is always isotropic over a finite field (this is a consequence of the
Chevalley-Warning theorem, cf. e.g., \cite[lemma 21.2.3 \&
  prop. 21.2.4]{FriedJarden2008}).
\end{rmk}

\begin{rmk}\label{remark-p-adic-fields}
Corollary \ref{main-result-reduction-number-fields} was deduced from
proposition \ref{main-result-valuation-rings} by applying it to the
localization $A = \mathcal{O}_{K,\mathfrak{p}}$ of $\mathcal{O}_K$
at $\mathfrak{p}$; alternatively but equivalently, one can apply it to
the \emph{completion} of $\mathcal{O}_K$ at $\mathfrak{p}$
(cf. \cite[prop. II.4.3]{Neukirch1999}), which is a $p$-adic field
(i.e., a finite extension of $\mathbb{Q}_p$).  For example,
\ref{main-result-valuation-rings} implies that whenever $p\equiv
3\pmod{4}$, we have $\chi((\mathbb{Q}_p)^2)
\leq\chi((\mathbb{F}_p)^2)$ and $\chi((\mathbb{Q}_p(\sqrt{p}))^2) \leq
\chi((\mathbb{F}_p)^2)$, the right-hand side of which can be computed
explicitly, and it is the latter inequality that will be used (for
$p=3$) in proposition \ref{chromatic-number-of-q-sqrt-3} below
(together with the fact that $\mathbb{Q}(\sqrt{3})$ is a subfield of
$\mathbb{Q}_3(\sqrt{3})$).
\end{rmk}

\thingy\label{remark-reduction-mod-4}
As mentioned in the introduction, it is known that $\chi(\mathbb{Q}^2)
= 2$: this is done by reducing mod $2$ but does not immediately follow
from \ref{main-result-reduction-number-fields} because $z_1^2 + z_2^2
= 0$ has solutions over $\mathbb{Z}/2\mathbb{Z}$; it follows, however,
from the following improvement of \ref{main-result-valuation-rings}
where instead of making an assumption over $A/\mathfrak{m}$ we make
one over $A/\mathfrak{m}^2$: indeed, there are no solutions to $z_1^2
+ z_2^2 = 0$ in $\mathbb{Z}/4\mathbb{Z}$ where $z_1$ or $z_2$ is odd,
so we still must have $\chi(\mathbb{Q}^2) \leq \chi((\mathbb{F}_2)^2)
= 2$.

\begin{prop}
Let $A$ be a valuation ring with valuation $v$: write $\mathfrak{m} :=
\{x\in A : v(x)>0\}$ for its (unique) maximal ideal and $K :=
\Frac(A)$ for the field of fractions of $A$.  Assume that the equation
$z_1^2 + \cdots + z_d^2 = 0$ has no solution mod $\mathfrak{m}^2$
except when all $z_i$ are in $\mathfrak{m}$.  Then there is a graph
homomorphism $\psi\colon \Gamma(K^d) \to \Gamma(A^d)$; in particular,
$\chi(K^d) = \chi(A^d) \leq \chi((A/\mathfrak{m}^2)^d) \leq
\chi((A/\mathfrak{m})^d)$.
\end{prop}

\begin{proof}
The proof is almost identical to that
of \ref{main-result-valuation-rings}: if $x_1^2 + \cdots + x_d^2 = 1$
with $x_i\in K$, then in fact $x_i \in A$ for all $i$, because if
$x_1^2 + \cdots + x_d^2 = 1$ and $v(x_i)<0$ for some $i$ and if $u$ is
such that $v(x_u)$ is the smallest, each $z_i := x_i/x_u$ belongs
to $A$, so does $1/x_u$, and $(x_1/x_u)^2 + \cdots + (x_d/x_u)^2 =
1/x_u^2$, and now $1/x_u^2$ belongs, in fact, to $\mathfrak{m}^2$, so
we get $\bar z_1^2 + \cdots + \bar z_d^2 = \bar 0$
in $A/\mathfrak{m}^2$ with $z_u = 1$ not belonging to $\mathfrak{m}$,
contradicting the hypothesis.  The rest is as previously.
\end{proof}

Sometimes we can push even further, as the following example shows
(which the author has not found the courage to try to formulate under
the most general auspices):

\begin{prop}\label{chromatic-number-of-q-sqrt-2}
The chromatic number $\chi(\mathbb{Q}(\sqrt{2})^2)$ of the plane with
coordinates in $\mathbb{Q}(\sqrt{2})^2$ is exactly $2$.  In fact, we
have $\chi(\mathbb{Q}_2(\sqrt{2})^2) = 2$.
\end{prop}

\begin{proof}
The lower bound is trivial: we wish to prove that
$\chi(\mathbb{Q}_2(\sqrt{2})^2) \leq 2$.

We will use the valuation $v$ on $K := \mathbb{Q}_2(\sqrt{2})$
normalized by: $v(\sqrt{2}) = \frac{1}{2}$, that is, $v(2) = 1$
(extending the valuation on $\mathbb{Q}_2$).

We first wish to show the following: if $x_1^2 + x_2^2 = 1$ in $K$,
then $v(x_i) \geq -\frac{1}{2}$ for both $i$.  Indeed, assume to the
contrary that $x_1^2 + x_2^2 = 1$ with $v(x_1) \leq v(x_2) \leq -1$.
Then $1 + (x_2/x_1)^2 = 1/x_1^2$ with $v(x_2/x_1) \geq 0$ and
$v(1/x_1^2) \geq 2$.  But the largest valuation that $1 + y^2$ can
take as $y$ ranges over $A := \mathbb{Z}_2[\sqrt{2}] = \{y \in K :
v(y) \geq 0\}$ is $\frac{3}{2}$, as can be checked by putting $y = a +
b \sqrt{2} + c\cdot 2$ with $a,b\in\{0,1\}$ and $c\in A$: then $1 +
y^2$ equals $1 + a^2 + b^2 \cdot 2 + ab\cdot(2\sqrt{2})$ plus terms of
valuation at least $2$ (i.e., zero mod $4$), and by considering all
four cases of values of $a,b$, we see that the largest valuation is
attained for $a=b=1$, namely $\frac{3}{2}$ (the valuation
of $2\sqrt{2}$).  This contradiction concludes the claim of the
paragraph.

Now as in the proof of \ref{main-result-valuation-rings}, we define an
equivalence relation on $K^2$, this time by putting $(x_1,x_2) \approx
(x'_1,x'_2)$ iff $v(x'_i-x_i) \geq -\frac{1}{2}$ for both $i$, i.e.,
$x'_1-x_1$ and $x'_2-x_2$ both belong to $\frac{1}{\sqrt{2}} A = \{y
\in K : v(y) \geq -\frac{1}{2}\}$.  From what we have just seen, each
edge of $\Gamma(K^d)$ connects two points in the same equivalence
class for $\approx$.  Again, by choosing a representative in each
equivalence class, we get $\chi(K^d) = \chi((\frac{1}{\sqrt{2}} A)^2)$
where $\chi((\frac{1}{\sqrt{2}} A)^2)$ is the chromatic number of the
induced subgraph of $\Gamma(K^d)$ consisting of vertices both of whose
coordinates belong to $\frac{1}{\sqrt{2}} A$ (=have
valuation $\geq-\frac{1}{2}$).  There is a slight difficulty, though,
in that $\frac{1}{\sqrt{2}} A$ is not a ring, so we can't immediately
apply \ref{observation-morphisms}.

To proceed with caution, let us define a new equivalence relation
$\equiv$ on $\frac{1}{\sqrt{2}} A$ by $x \equiv y$ iff $v(y-x) \geq
\frac{1}{2}$, that is, $y = x+\varepsilon$ with $v(\varepsilon) \geq
\frac{1}{2}$.  Note that $x \equiv y$ and $x' \equiv y'$ imply $x+x'
\equiv y+y'$, that is, $(\frac{1}{\sqrt{2}} A)/\equiv$ is an additive
group; in fact, it is the Klein four-element group represented by
$\{0,1,\frac{1}{\sqrt{2}}, \frac{1}{\sqrt{2}}+1\}$ where every element
has order $2$.  Now, if $y = x+\varepsilon$ with $v(\varepsilon) \geq
\frac{1}{2}$, we have $y^2 = x^2 + 2\varepsilon x + \varepsilon^2$ and
$v(2\varepsilon x) = 1 + v(\varepsilon) + v(x) \geq 1$ (because $v(x)
\geq -\frac{1}{2}$) and $v(\varepsilon^2) \geq 1$, so $v(y^2 - x^2)
\geq 1$.  We further define $\equiv$ on $(\frac{1}{\sqrt{2}} A)^2$ by
$(x_1,x_2) \equiv (y_1,y_2)$ iff $x_1 \equiv y_1$ and $x_2 \equiv y_2$
and then we therefore have $v((y_1^2+y_2^2) - (x_1^2+x_2^2)) \geq 1$.
In particular, if $v(x_1^2+x_2^2 - 1) \geq 1$, then $v(y_1^2+y_2^2 -
1) \geq 1$ also, and we define a graph $\Gamma((\frac{1}{\sqrt{2}}
A)^2/\equiv)$ with set of vertices $(\frac{1}{\sqrt{2}} A)^2/\equiv$
by connecting $(x_1,x_2)$ and $(x'_1,x'_2)$ with an edge whenever
$(x'_1-x_1)^2 + (x'_2-x_2)^2$ is equal to $1$ mod $2$ (i.e.,
$v((x'_1-x_1)^2 + (x'_2-x_2)^2 - 1) \geq 1$).  From what we have just
seen, this does not depend on the equivalence class of $x$ or $x'$
for $\equiv$ (we use the fact that if $x\equiv y$ and $x'\equiv y'$
then $(x'-x) \equiv (y'-y)$).  Since reduction mod $\equiv$ is
obviously graph homomorphism, the chromatic number
$\chi((\frac{1}{\sqrt{2}} A)^2/\equiv)$ of the graph just defines is
at least $\chi((\frac{1}{\sqrt{2}} A)^2)$, which we saw is equal to
$\chi(K^2)$.

But $\Gamma((\frac{1}{\sqrt{2}} A)^2/\equiv)$ is a finite graph.  It
has sixteen vertices, which are represented by
$\{0,1,\frac{1}{\sqrt{2}}, \frac{1}{\sqrt{2}}+1\}^2$.  For
conciseness, we will write $U$ for $\frac{1}{\sqrt{2}}$ and $V$ for
$\frac{1}{\sqrt{2}}+1$ and concatenate both coordinates of the
vertices.  The vertices adjacent to the origin $00$ are: $01$, $10$,
$UU$ (because $(\frac{1}{\sqrt{2}})^2 + (\frac{1}{\sqrt{2}})^2 = 1$)
and $VV$ (because $(\frac{1}{\sqrt{2}}+1)^2 + (\frac{1}{\sqrt{2}}+1)^2
= 1 + 2 + 2\sqrt{2}$), and other adjacency relations are obtained by
translation (remembering that each coordinate is in a Klein
four-group).  A coloring of the graph $\Gamma((\frac{1}{\sqrt{2}}
A)^2/\equiv)$ with two colors is obtained by giving one color to the
eight vertices $00$, $11$, $U0$, $V1$, $U1$, $V0$, $UV$, $VU$, and the
other color to the eight other vertices.
\end{proof}

\section{Explicit values and bounds for certain fields}\label{section-explicit-values}

\thingy\label{chromatic-number-of-f3} It is easy to see that
$\chi((\mathbb{F}_3)^2) = 3$: it is no less because the points
$(0,0)$, $(1,0)$ and $(2,0)$ form a triangle, and it is no more
because one can color $(u,v)$ with color $u+v$ (mod $3$).

\begin{prop}\label{chromatic-number-of-q-sqrt-3}
The chromatic number $\chi(\mathbb{Q}(\sqrt{3})^2)$ of the plane with
coordinates in $\mathbb{Q}(\sqrt{3})^2$ is exactly $3$.
\end{prop}

\begin{proof}
By \ref{embedding-triangle}, we know that
$\chi(\mathbb{Q}(\sqrt{3})^2) \geq 3$.

On the other hand, the ring of integers of $\mathbb{Q}(\sqrt{3})$ is
$\mathbb{Z}[\sqrt{3}]$, and $\sqrt{3}$ generates a prime ideal with
residue field $\mathbb{F}_3$.  Now $\chi((\mathbb{F}_3)^2) = 3$ as we
have just noted, so \ref{main-result-reduction-number-fields} gives
$\chi(\mathbb{Q}(\sqrt{3})^2) \leq 3$.
\end{proof}

The following proposition exhibits a situation where the chromatic
number can be computed exactly but is not equal to the clique
number\footnote{The clique number $\omega(G)$ of a graph $G$ is the
  largest $n$ for which there exists a graph homomorphism $K_n \to G$,
  or $\infty$ if there is no largest $n$.  Evidently, $\omega(G) \leq
  \chi(G)$.  When $\omega(G) = 2$, the graph is said to be
  triangle-free.}:

\begin{prop}\label{chromatic-number-of-q-sqrt-7}
If $K = \mathbb{Q}(\sqrt{7})$, then $\Gamma(K^2)$ is triangle-free,
but its chromatic number $\chi(K^2)$ is still exactly $3$.
\end{prop}
\begin{proof}
First we check that $\Gamma(K^2)$ is triangle-free.  Assume it
contains a triangle $u,v,w$.  By translating, we can assume that $u$
is the origin.  Let us explain why we can assume that $v = (1,0)$:
\textit{a priori} we have $v = (v_1,v_2)$ with $v_1^2 + v_2^2 = 1$,
but then the matrix
\[
\left(\begin{matrix}v_1&v_2\\-v_2&v_1\end{matrix}\right)
\]
(acting from the left on column vectors) has values in $K$, preserves
the quadratic form $x_1^2 + x_2^2$, and takes $v$ to $(1,0)$.
(Essentially, we are saying that once we have a vector of unit norm
with values in $K$, the rotation taking that vector to $(1,0)$ also
has values in $K$.)  So we are left with $w = (w_1,w_2)$ which
satisfies $w_1^2 + w_2^2 = 1$ and $(w_1-1)^2 + w_2^2 = 1$, giving
$w_1=\frac{1}{2}$ and $w_2^2 = \frac{3}{4}$.  Since $3$ is not a
square in $K$, the latter has no solution and there is no triangle.

To prove the lower bound $\chi(K^2) \geq 3$, we construct an odd cycle
in $\Gamma(K^2)$: namely, $(0,0)$,
$(\frac{1}{8},\frac{3\sqrt{7}}{8})$, $(\frac{1}{4},0)$,
$(\frac{3}{8},\frac{3\sqrt{7}}{8})$, $(\frac{1}{2},0)$,
$(\frac{5}{8},\frac{3\sqrt{7}}{8})$, $(\frac{3}{4},0)$,
$(\frac{7}{8},\frac{3\sqrt{7}}{8})$, $(1,0)$.  It is straightforward
to check that two consecutive vertices of these nine (taken
cyclically) are at distance $1$, so we have a graph homomorphism from
$C_9$ to $\Gamma(K^2)$.  This shows $\chi(K^2) \geq \chi(C_9) = 3$.

As for the upper bound: the ring of integers of $\mathbb{Q}(\sqrt{7})$
is $\mathbb{Z}[\sqrt{7}]$, and $\sqrt{7} - 2$ generates a prime ideal
with residue field $\mathbb{F}_3$.  Now $\chi((\mathbb{F}_3)^2) = 3$
(again, \ref{chromatic-number-of-f3}), so
\ref{main-result-reduction-number-fields} gives
$\chi(\mathbb{Q}(\sqrt{7})^2) \leq 3$.
\end{proof}

The previous examples are all subfields of $\mathbb{R}$ for which we
have the upper bound $\chi(\mathbb{R}^2) \leq 7$ well-known in the
Hadwiger-Nelson problem.  But the reasoning used also works for
certain non real number fields:

\begin{prop}
If $K = \mathbb{Q}(\sqrt{-5})$, then $\Gamma(K^2)$ is triangle-free,
but its chromatic number $\chi(K^2)$ is still exactly $3$.
\end{prop}
\begin{proof}
The proof that $\Gamma(K^2)$ is triangle-free is the same as
in \ref{chromatic-number-of-q-sqrt-7}: again, $3$ is not a square
in $K$.

To prove the lower bound, use the following odd cycle: $(0,0)$,
$(1,0)$, $(2,0)$, $(3,0)$ and $(\frac{3}{2},\frac{\sqrt{-5}}{2})$:
this gives a graph homomorphism from $C_5$ to $\Gamma(K^2)$.

For the upper bound: the ring of integers of $\mathbb{Q}(\sqrt{-5})$
is $\mathbb{Z}[\sqrt{-5}]$, and $(3,\, \sqrt{-5}+1)$ is a prime ideal
with residue field $\mathbb{F}_3$.  So the conclusion follows once
again by \ref{main-result-reduction-number-fields}
and \ref{chromatic-number-of-f3}.
\end{proof}

The previous examples all used the fact (\ref{chromatic-number-of-f3})
that $\chi((\mathbb{F}_3)^2) = 3$.  Unfortunately, there aren't that
many finite fields that \emph{can} be used to produce a meaningful
upper bound.  Here, however, is an example of a subfield
of $\mathbb{R}$ where we can give a lower bound that is greater
than $3$ and an upper bound that is better than the standard upper
bound on $\chi(\mathbb{R}^2)$ (viz. $7$):

\begin{lem}\label{chromatic-number-of-f5}
We have\footnote{In fact, $\chi((\mathbb{F}_{11})^2) = 5$, but we will
  neither prove nor use this.} $\chi((\mathbb{F}_{11})^2) \leq 5$.
\end{lem}
\begin{proof}
Consider the following table:
\[
\begin{array}{ccccccccccc}
 \mathtt{3}  &  \mathtt{1}  & \llap{$*$}\mathtt{0}\rlap{$*$} &  \mathtt{2}  &  \mathtt{1}  &  \mathtt{2}  &  \mathtt{3}  &  \mathtt{4}  & \llap{$*$}\mathtt{2}\rlap{$*$} &  \mathtt{0}  &  \mathtt{1}  \\
 \mathtt{1}  &  \mathtt{2}  &  \mathtt{1}  &  \mathtt{0}  &  \mathtt{4}  &  \mathtt{1}  &  \mathtt{2}  &  \mathtt{3}  &  \mathtt{4}  &  \mathtt{3}  &  \mathtt{2}  \\
\llap{$*$}\mathtt{2}\rlap{$*$} &  \mathtt{1}  &  \mathtt{2}  &  \mathtt{3}  &  \mathtt{0}  &  \mathtt{2}  &  \mathtt{1}  &  \mathtt{2}  &  \mathtt{3}  &  \mathtt{4}  & \llap{$*$}\mathtt{0}\rlap{$*$} \\
 \mathtt{0}  &  \mathtt{2}  &  \mathtt{3}  &  \mathtt{1}  &  \mathtt{3}  &  \mathtt{4}  &  \mathtt{0}  &  \mathtt{4}  &  \mathtt{1}  &  \mathtt{3}  &  \mathtt{4}  \\
 \mathtt{4}  &  \mathtt{0}  &  \mathtt{1}  &  \mathtt{2}  &  \mathtt{1}  & \llap{$*$}\mathtt{3}\rlap{$*$} &  \mathtt{4}  &  \mathtt{0}  &  \mathtt{2}  &  \mathtt{1}  &  \mathtt{2}  \\
 \mathtt{3}  &  \mathtt{4}  &  \mathtt{0}  &  \mathtt{1}  & \llap{$*$}\mathtt{2}\rlap{$*$} & \llap{$\bullet$\!}\mathtt{1}\rlap{\!$\bullet$} & \llap{$*$}\mathtt{3}\rlap{$*$} &  \mathtt{4}  &  \mathtt{1}  &  \mathtt{2}  &  \mathtt{1}  \\
 \mathtt{1}  &  \mathtt{2}  &  \mathtt{3}  &  \mathtt{4}  &  \mathtt{0}  & \llap{$*$}\mathtt{3}\rlap{$*$} &  \mathtt{1}  &  \mathtt{0}  &  \mathtt{4}  &  \mathtt{0}  &  \mathtt{2}  \\
 \mathtt{2}  &  \mathtt{1}  &  \mathtt{4}  &  \mathtt{3}  &  \mathtt{4}  &  \mathtt{0}  &  \mathtt{2}  &  \mathtt{3}  &  \mathtt{0}  &  \mathtt{3}  &  \mathtt{4}  \\
\llap{$*$}\mathtt{4}\rlap{$*$} &  \mathtt{2}  &  \mathtt{3}  &  \mathtt{4}  &  \mathtt{3}  &  \mathtt{4}  &  \mathtt{0}  &  \mathtt{2}  &  \mathtt{3}  &  \mathtt{4}  & \llap{$*$}\mathtt{0}\rlap{$*$} \\
 \mathtt{0}  &  \mathtt{4}  &  \mathtt{1}  &  \mathtt{2}  &  \mathtt{1}  &  \mathtt{0}  &  \mathtt{4}  &  \mathtt{0}  &  \mathtt{2}  &  \mathtt{3}  &  \mathtt{2}  \\
 \mathtt{4}  &  \mathtt{0}  & \llap{$*$}\mathtt{4}\rlap{$*$} &  \mathtt{1}  &  \mathtt{2}  &  \mathtt{3}  &  \mathtt{0}  &  \mathtt{3}  & \llap{$*$}\mathtt{0}\rlap{$*$} &  \mathtt{1}  &  \mathtt{3}  \\
\end{array}
\]
it is an $11\times 11$ array of numbers from $\mathtt{0}$ to
$\mathtt{4}$ (representing five colors): if the rows and columns are
identified cyclically with elements of $\mathbb{F}_{11}$ (the starting
row/column and the order in which they are read is, of course,
irrelevant), then one can check that two squares whose cyclic row
distance $u_1$ and cyclic column distance $u_2$ are related by $u_1^2
+ u_2^2 = 1$ never contain the same number; to make it perhaps easier
to check this fact by hand, we have marked with asterisks the $12$
squares which are connected by an edge to the central one (itself
marked with bullets): so one should check that no square marked with
an asterisk contains the same number as that marked with bullets, and
similarly for any cyclic translation of this pattern.
\end{proof}

\begin{prop}\label{chromatic-number-of-q-sqrt-3-and-11}
If $K = \mathbb{Q}(\sqrt{3},\sqrt{11})$, then
$4 \leq \chi(K^2) \leq 5$.
\end{prop}
\begin{proof}
The lower bound follows from \ref{embedding-mosers-spindle}.

The upper bound, obtained by
\ref{main-result-reduction-number-fields}, uses the fact that
$\chi((\mathbb{F}_{11})^2) \leq 5$ by \ref{chromatic-number-of-f5},
and that the ideal generated by $-\frac{5}{2} - \frac{\sqrt{3}}{2} +
\frac{\sqrt{11}}{2} + \frac{\sqrt{33}}{2}$ has residue
field $\mathbb{F}_{11}$ (note that $11$ factors as $(23 + 4\sqrt{33})
\times \big(-\frac{5}{2} - \frac{\sqrt{3}}{2} + \frac{\sqrt{11}}{2} +
\frac{\sqrt{33}}{2}\big)^2 \times \big(-\frac{5}{2} +
\frac{\sqrt{3}}{2} - \frac{\sqrt{11}}{2} +
\frac{\sqrt{33}}{2}\big)^2$, where $23 + 4\sqrt{33}$ is a unit having
inverse $23 - 4\sqrt{33}$).
\end{proof}

\section{Remarks on algebraically and real closed fields}\label{section-closed-fields}

\thingy\label{remarks-algebraically-closed-fields}
In the introduction, we mention the question of computing
$\chi(\mathbb{C}^d)$.  In fact, for algebraically closed fields $E$,
the value of $\chi(E^d)$ depends only on the characteristic $p$ of $E$
and not on the field $E$ itself.  Indeed, the finite graphs $G$ for
which there exists a graph homomorphism $G \to \Gamma(E^d)$ with $E$
algebraically closed depends only on the characteristic of $E$ (and,
of course, on $d$).

Here is one way of seeing this fact: if $E$ is any field, for any
finite graph $G$ with $N$ vertices, saying that there does \emph{not}
exist a graph homomorphism $G \to \Gamma(E^d)$ means that the set of
$(N d)$-tuples of elements $(x_{\gamma,i})$ of $E$, indexed by the
vertices $\gamma$ of $G$ and $1\leq i\leq d$, subject to the relations
$(x_{\gamma',1}-x_{\gamma,1})^2 + \cdots +
(x_{\gamma',d}-x_{\gamma,d})^2 - 1 = 0$ for each edge
$\{\gamma,\gamma'\}$ of $G$, is empty.  Now if $E$ is algebraically
closed, by Hilbert's Nullstellensatz, this is equivalent to saying
that the polynomials $h_{\{\gamma,\gamma'\}} :=
(x_{\gamma',1}-x_{\gamma,1})^2 + \cdots +
(x_{\gamma',d}-x_{\gamma,d})^2 - 1$ (again, where $\{\gamma,\gamma'\}$
ranges over the set $\mathrm{E}(G)$ of edges of $G$) generate the unit
ideal of the polynomial ring $E[(x_{\gamma,i})]$ in $N d$ variables
(i.e., that we can write $\sum_{\{\gamma,\gamma'\}\in \mathrm{E}(G)}
g_{\{\gamma,\gamma'\}} h_{\{\gamma,\gamma'\}} = 1$ for some
$g_{\{\gamma,\gamma'\}} \in E[(x_{\gamma,i})]$).  But this depends
only on the characteristic.  Indeed, if there is a combination
$\sum_{\{\gamma,\gamma'\}\in \mathrm{E}(G)} g_{\{\gamma,\gamma'\}}
h_{\{\gamma,\gamma'\}} = 1$ then, for any $\mathbb{Z}$-linear form
$\lambda \colon E \to F$ such that $\lambda(1) = 1$, we have
$\sum_{\{\gamma,\gamma'\}\in \mathrm{E}(G)}
\lambda(g_{\{\gamma,\gamma'\}})\, h_{\{\gamma,\gamma'\}} = 1$ where
$\lambda(g)$ means $\lambda$ is applied to all coefficients of $g$
(note that $h_{\{\gamma,\gamma'\}}$ has integer coefficients!); now if
$E$ and $F$ are two algebraically closed fields of the same
characteristic, we can obviously find such $\lambda$.

Another possible proof notes that the statement that there does, or
does not, exist a graph homomorphism $G \to \Gamma(E^d)$ is a
first-order statement when interpreted in the field $E$, and the
(first-order) theory of algebraically closed fields of fixed
characteristic is complete, i.e., all of its models are elementarily
equivalent, so the validity of a first-order statement does not depend
on the model.  (Cf. \cite[theorem 6.4]{Poizat2000} or
\cite[chapter 9]{FriedJarden2008}.)

One consequence of the above remarks is that $\chi(\mathbb{C}^d) =
\chi((\mathbb{Q}^{\mathrm{alg}})^d)$ where $\mathbb{Q}^{\mathrm{alg}}$
stands for the algebraic closure of $\mathbb{Q}$.  Furthermore,
$\chi(\mathbb{C}^d)$ is the greatest value of the $\chi(E^d)$ for all
fields $E$ of characteristic zero, just as
$\chi((\mathbb{F}_p^{\mathrm{alg}})^d)$ is the greatest value of the
$\chi(E^d)$ for all fields of characteristic $p$.  Another fact worthy
of note is that, for any $n,d$ the fact that $\chi(\mathbb{C}^d)\geq
n$, if true, is provable (by enumerating all finite graphs $G$ until
one finds one with chromatic number $\geq n$ and which admits a
homomorphism to $\Gamma(\mathbb{C}^d)$, a fact which can be tested
using the Nullstellensatz and Gröbner bases, or some other decision
procedure for algebraically closed fields).

One could argue from the above presentation that, from an algebraic
point of view, the question of computing $\chi(\mathbb{C}^d)$, or more
generally, deciding which finite graphs admit a homomorphism to
$\Gamma(E^d)$ for an algebraically closed field $E$ of a given
characteristic, is more fundamental and perhaps more interesting than
the case $\chi(\mathbb{R}^d)$ of real closed field
(cf. \ref{remarks-real-closed-fields} below) considered by the
classical Hadwiger-Nelson problem.  Certainly, if it turns out that
$\chi(\mathbb{C}^2) = 4$, this would be a more profound result than
$\chi(\mathbb{R}^2) = 4$ (which it implies).

\thingy We have explained above why the finite graphs which admit a
homomorphism to $\Gamma(E^d)$ for $E$ an algebraically closed field
depend only on ($d$ and) the characteristic of $E$.  We can state the
following fact in comparing characteristic $0$ to the others:

If $G$ is a finite graph that admits a graph homomorphism $G \to
\Gamma((\mathbb{F}_p^{\mathrm{alg}})^d)$ for infinitely many
primes $p$ (where $\mathbb{F}_p^{\mathrm{alg}}$ refers to the
algebraic closure $\bigcup_{n=1}^{\infty} \mathbb{F}_{p^n}$ of
$\mathbb{F}_p$), then there is a graph homomorphism $G \to
\Gamma(\mathbb{C}^d)$.  Equivalently: if a given finite graph $G$
admits a homomorphism to $\Gamma(K^d)$ for fields $K$ of arbitrarily
large finite characteristic, then it admits one to a field of
characteristic zero (which, as we have seen, can be chosen to be
$\mathbb{Q}^{\mathrm{alg}}$ or $\mathbb{C}$).

An algebraically minded proof proceeds as follows: if there is no
graph homomorphism $G \to \Gamma(\mathbb{C}^d)$, then as in the
discussion above, we can write $\sum_{\{\gamma,\gamma'\}\in
  \mathrm{E}(G)} g_{\{\gamma,\gamma'\}} h_{\{\gamma,\gamma'\}} = 1$
for some $g_{\{\gamma,\gamma'\}} \in \mathbb{C}[(x_{\gamma,i})]$
labeled by the edges of $G$, and $h_{\{\gamma,\gamma'\}} :=
(x_{\gamma',1}-x_{\gamma,1})^2 + \cdots +
(x_{\gamma',d}-x_{\gamma,d})^2 - 1$.  Using some $\mathbb{Z}$-linear
form $\lambda \colon \mathbb{C} \to \mathbb{Q}$ such that
$\lambda(1)=1$, we can even find the $g_{\{\gamma,\gamma'\}}$ with
coefficients in $\mathbb{Q}$.  Now only finitely many primes divide
the denominators of these $g_{\{\gamma,\gamma'\}}$, and reducing
modulo any other $p$ gives a relation of the same sort that precludes
the existence of $G \to \Gamma((\mathbb{F}_p^{\mathrm{alg}})^d)$.

A more logically minded proof of the same thing proceeds by noting
that the theory of algebraically closed fields of characteristic $0$
consists of infinitely many axioms, any finite number of which are
valid for sufficiently large characteristics.  So if the
inexistence\footnote{This also works for the \emph{existence} of such
  a morphism, but the conclusion is subsumed in
  proposition \ref{comparison-of-characteristics} anyway.} of a graph
homomorphism $G \to \Gamma(\mathbb{C}^d)$ can be proved from these
axioms, it can be proved from finitely many of them, giving the
desired conclusion.

This fact does not seem to have any exploitable consequence on the
chromatic number, but here is a converse that does:

\begin{prop}\label{comparison-of-characteristics}
Let $G$ be a finite graph that admits a graph homomorphism $G \to
\Gamma(\mathbb{C}^d)$.  Then there is a graph homomorphism $G \to
\Gamma((\mathbb{F}_p)^d)$ for a set $\mathscr{P}$ of prime numbers
having positive density.  In particular, we have $\chi(\mathbb{C}^d)
\leq \limsup_{p\to+\infty} \chi((\mathbb{F}_p)^d)$.
\end{prop}
\begin{proof}
We know that $G$ admits a homomorphism to
$\Gamma((\mathbb{Q}^{\mathrm{alg}})^d)$.  So (since the vertices of
the image generate a finite extension) there is one to $\Gamma(K^d)$
for some finite extension $K$ of $\mathbb{Q}$ (=number field).  Given
such a graph homomorphism, there are only finitely many primes
$\mathfrak{p}$ of $\mathcal{O}_K$ such that the coordinates of the
image vertices are not all integers at $\mathfrak{p}$ (``have
denominators in $\mathfrak{p}$'').  Furthermore, by the Čebotarëv
density theorem (\cite[theorem 13.4]{Neukirch1999} or
\cite[theorem 6.3.1]{FriedJarden2008}), there exists a set
$\mathscr{P}$ of primes with positive density such that $p \in
\mathscr{P}$ iff $p$ is unramified in $K$ and there is a prime
$\mathfrak{p}$ of $\mathcal{O}_K$ lying over $p$ and having degree $1$
(i.e., same residue field $\mathbb{F}_p$).  (Precisely, if $\Sigma$ is
the Galois group over $\mathbb{Q}$ of the Galois closure of $K$, then
the density of $\mathscr{P}$ is the proportion of elements of $\Sigma$
whose conjugacy class meets the fixator of $K$.)  So possibly removing
finitely many elements from $\mathscr{P}$, we obtained the required
conditions: $G$ admits a homomorphism to
$\Gamma((\mathcal{O}_{K,\mathfrak{p}})^d)$ for some $\mathfrak{p}$
such that $\mathcal{O}_{K,\mathfrak{p}} / \mathfrak{p} \cong
\mathbb{F}_p$.
\end{proof}

Note that we do not need to use \ref{main-result-valuation-rings}
here: we are considering an infinite set of primes, so one simply
excludes those in which there are denominators.  Note that the above
result implies a bound for the classical Hadwiger-Nelson problem,
viz. $\chi(\mathbb{R}^d) \leq \limsup_{p\to+\infty}
\chi((\mathbb{F}_p)^d)$, where each term of the sequence in the right
hand side is finitely computable (although this bound is quite
possibly infinite).

\thingy\label{remarks-real-closed-fields} We can say for real closed
fields much of what we said
in \ref{remarks-algebraically-closed-fields} above for algebraically
closed fields.  Specifically, the finite graphs $G$ for which there
exists a graph homomorphism $G \to \Gamma(E^d)$ with $E$ real closed
do not depend on $E$, and in particular, the value of $\chi(E^d)$ is
the same for all real closed field $E$ (it depends only on $d$).  This
time, the proof invokes Tarski's theorem on the decidability of the
first-order theory of real closed fields
(\cite[theorem 6.41]{Poizat2000}).  One consequence is that
$\chi(\mathbb{R}^d) = \chi((\mathbb{Q}^{\mathrm{r-alg}})^d)$ where
$\mathbb{Q}^{\mathrm{r-alg}}$ stands for the real closure
of $\mathbb{Q}$ (which can be seen as the set of real algebraic
numbers)\footnote{This answers \cite[problem 11.1]{Soifer2009}, but
  the question is perhaps misstated since the MathSciNet review of
  \cite{BendaPerles2000} (the present author does not have access to
  the paper itself) suggests that the remark above is already
  contained there.}.  Another is that, for any $n,d$ the fact that
$\chi(\mathbb{R}^d)\geq n$, if true, is provable (by enumerating all
finite graphs $G$ until one finds one with chromatic number $\geq n$
and which admits a homomorphism to $\Gamma(\mathbb{R}^d)$, a fact
which can be tested using some other decision procedure for real
closed fields).

In particular, if the answer to the classical Hadwiger-Nelson problem
turns out to be $\chi(\mathbb{R}^2) = 7$, then this fact is provable.

\section{Remarks on changing the quadratic form}\label{section-quadratic-form}

We can generalize \ref{definition-main-graph} as follows:

\begin{dfn}
Let $E$ be any field (or even any commutative ring) and $q$ a
quadratic form in $d \geq 1$ variables over $E$.  We define a graph
$\Gamma(E^d, q)$ as follows: vertices of $\Gamma(E^d, q)$ are
$d$-tuples from $E$, with an edge between $(x_1,\ldots,x_d)$ and
$(x'_1,\ldots,x'_d)$ whenever $q(x'-x) = 1$.  We write $\chi(E^d, q) =
\chi(\Gamma(E^d, q))$ for the chromatic number of this graph
$\Gamma(E^d, q)$ (possibly $+\infty$).
\end{dfn}

The case considered in \ref{definition-main-graph} is that where $q =
x_1^2 + \cdots + x_d^2$.

\thingy If $E$ is a field and the quadratic form $q$ is degenerate,
meaning that there is a nontrivial subspace $V$ of $E^d$ (the largest
of which is called $\ker(q)$) such that $q(x+y) = q(x)$ for all $x\in
E^d$ and $y\in V$, then we can color $E^d$ with a certain number of
colors by coloring (the quotient vector space) $E^d/V$, and it is easy
to see that $\chi(E^d, q) = \chi(E^d/V, q)$ where the second $q$
refers to the obviously defined quadratic field on $E^d/V$, and it is
nondegenerate.  So we can always assume (up to a change in $d$) that
$q$ is nondegenerate.

\thingy If $E$ is the field $\mathbb{R}$ of reals, or more generally a
real closed field, then Sylvester's law of inertia states that any
nondegenerate quadratic field on $E^d$ is equivalent to $x_1^2 +
\cdots + x_s^2 \, - x_{s+1}^2 - \cdots - x_d^2$ for some $0\leq s \leq
d$: the pair $(s,d-s)$ is called the \emph{signature} of the quadratic
form.  The graph $\Gamma(E^d, q)$ (up to isomorphism), and its
chromatic number $\chi(E^d, q)$ only depend on this signature.  But
note that if $E$ is a subfield of $\mathbb{R}$, in general, not all
quadratic forms over $E$ that are equivalent to $x_1^2 + \cdots +
x_d^2$ over $\mathbb{R}$ will be equivalent to it over $E$: and the
computation of any such $\chi(E^d, q)$ might legitimately be
considered an analogue of the Hadwiger-Nelson problem with
coefficients in $E$.

\thingy\label{lorentzian-case}
In the case of rank $d=2$ over the reals (or more generally a real
closed field), the only nondegenerate quadratic form other than the
Euclidean $x_1^2 + x_2^2$ is the \emph{Lorentzian} (or
\emph{Minkowskian}) quadratic form, viz. $q_L := x_1^2 - x_2^2$.  This
quadratic form is of great importance in special relativity (if $x_1$
is the time coordinate and $x_2$ the space coordinate, then
$\sqrt{q(x-y)}$ defines the proper-time separation of the events $x$
and $y$).

The question of computing $\chi(\mathbb{R}^2, q_L)$, or even just
deciding whether it is finite, seems an interesting one (to which the
present author does not know the answer), and it would shed light on
how to handle the ``isotropic'' case ($q_L$ has nontrivial zeros).
One thing that can be said is that $\chi(\mathbb{R}^2, q_L) \leq
\chi(\mathbb{C}^2)$ (since all nondegenerate quadratic forms
over $\mathbb{C}$ are equivalent); conversely, $\chi(\mathbb{C}^2)
\leq \chi(\mathbb{R}^4,\; x_1^2 + x_2^2 - x_3^3 - x_4^2)$ is easily seen
by separating complex numbers into real and imaginary parts, so the
two problems of computing $\chi(\mathbb{R}^2, q_L)$ and
$\chi(\mathbb{C}^2)$ are intimately related.  We can at least say
this:

\begin{prop}
If $q_L = x_1^2 - x_2^2$ is the ``Lorentzian'' quadratic form on
$\mathbb{R}^2$, then $\Gamma(\mathbb{R}^2,q_L)$ has cycles of any
order $\geq 3$ but no triangle.  In particular, $\chi(\mathbb{R}^2,
q_L) \geq 3$.
\end{prop}
\begin{proof}
First we show that $\Gamma(\mathbb{R}^2,q_L)$ has no triangle.  Before
we do this, we define the \emph{causal partial order} on
$\mathbb{R}^2$ as follows: we say that $x \mathrel{<}_{\mathrm{caus}}
y$ for $x,y \in \mathbb{R}^2$ when $q(y-x)>0$ and $x_1 < y_1$, or
alternatively, $|y_2 - x_2| < y_1 - x_1$.  This causal partial order
defines an orientation on the edges of $\Gamma(\mathbb{R}^2,q_L)$
(orient an edge $(x,y)$ from $x$ to $y$ when $x
\mathrel{<}_{\mathrm{caus}} y$, which here just means $x_1 < y_1$).
Now assume $x,y,z$ is a triangle.  By permuting, we can assume $x
\mathrel{<}_{\mathrm{caus}} y \mathrel{<}_{\mathrm{caus}} z$.  By
translating, we can assume $x=(0,0)$.  By applying the ``Lorentz
group'' $\{T_\eta\colon (u_1,u_2) \mapsto (u_1 \cosh\eta + u_2
\sinh\eta,\; u_1 \sinh\eta + u_2 \cosh\eta)\}$, which preserves the
quadratic form $q_L$, we can assume $y = (1,0)$ (just take $\eta =
-\arctanh(y_2/y_1)$).  Now we have simultaneously $z_1^2 - z_2^2 = 1$
and $(z_1-1)^2 - z_2^2 = 1$, so simultaneously $z_1 = \sqrt{1+z_2^2}$
and $z_1 = 1+\sqrt{1+z_2^2}$, a contradiction.  This shows that there
are no triangles.

To construct a $(k+2)$-cycle for any $k\geq 3$, consider the $k+1$
points $(i,0)$ for $0\leq i\leq k$, together with $(\frac{k}{2},
\frac{\sqrt{k^2-4}}{2})$, the latter being adjacent to both $(0,0)$
and $(k,0)$.  For a $4$-cycle, consider for example $(0,0)$, $(1,0)$,
$(\frac{9}{4},\frac{3}{4})$ and $(\frac{5}{4},\frac{3}{4})$.

The conclusion on the chromatic number follows from the existence of
an odd cycle.
\end{proof}

\thingy Proposition \ref{main-result-valuation-rings} was stated for
$\Gamma(K^d)$ and $\Gamma(\kappa^d)$ for simplicity, but the proof
carries over exactly to $\Gamma(K^d, q)$ and $\Gamma(\kappa^d,\bar q)$
if $q$ is a quadratic form in $d$ variables with coefficients in $A$
and $\bar q$ is its reduction mod $\mathfrak{m}$.  The statement is
then:

Let $A$ be a valuation ring with valuation $v$: write $\mathfrak{m} :=
\{x\in A : v(x)>0\}$ for its (unique) maximal ideal, $\kappa :=
A/\mathfrak{m}$ the residue field, and $K := \Frac(A)$ for the field
of fractions of $A$.  Assume that $q$ is a quadratic form in $d$
variables with coefficients in $A$ such that the quadratic form $\bar
q$ obtained by reducing these coefficients mod $\mathfrak{m}$ is
\emph{anisotropic} over $\kappa$ (that is, $\bar q(z_1,\ldots,z_d) =
0$ has no solution in $\kappa$ other than the trivial
$(z_1,\ldots,z_d) = (0,\ldots,0)$).  Then there is a graph
homomorphism $\psi\colon \Gamma(K^d, q) \to \Gamma(A^d, q)$.  In
particular, $\chi(K^d, q) = \chi(A^d, q) \leq \chi(\kappa^d, \bar q)$.

\section{Remarks on the role of the axiom of choice}\label{section-choice}

The axiom of choice is used in several places in the results above:
remark \ref{fields-of-charac-2} uses it to produce an
$\mathbb{F}_2$-linear form on a field $E$ of characteristic $2$ that
takes the value $1$ at $1$, and more importantly,
proposition \ref{main-result-valuation-rings} uses it to select a
representative $\xi_C$ from each equivalence class $C$ of $K^d/A^d$.
In the absence of the axiom of choice, we can still say certain
things, however:

\thingy If, instead of working with the chromatic number $\chi(G)$ of
a graph, we work with the ``finite-limit-chromatic number''
$\chi_{\mathrm{fin}}(G)$, which is \emph{defined} as the upper bound
of the $\chi(G_0)$ for all finite subgraphs $G_0$ of $G$, making the
De Bruijn-Erdős theorem trivially true, then the results of sections
\ref{section-generalities} to \ref{section-explicit-values} of this
paper hold, in the absence of Choice, for $\chi_{\mathrm{fin}}$
instead of $\chi$ (because only finitely many choices have to be
made).

Note that the question of computing
$\chi_{\mathrm{fin}}(\mathbb{R}^2)$ is precisely the same as that of
computing $\chi(\mathbb{R}^2)$ in the presence of the axiom of choice.
Furthermore, since the statement ``$\chi_{\mathrm{fin}}(\mathbb{R}^2)
= n$'' is an \emph{arithmetical} one (i.e., one that can be stated in
the language of first-order arithmetic: namely, the one which states
that every finite unit-distance graph with real algebraic coordinates
can be colored with $n$ colors and at least one requires this number
of colors), its truth value does not, in fact, depend on the axiom of
choice (because the Gödel constructible universe $L$ has the same
integers, so the same true arithmetical statements as the real
universe $V$ of set theory).  One might therefore argue that the
``right'' Hadwiger-Nelson problem in the absence of choice concerns
the value of $\chi_{\mathrm{fin}}(\mathbb{R}^2)$, not
$\chi(\mathbb{R}^2)$ (which might be ``artificially higher'' because
certain colorings are not available in the absence of choice): the
value of $\chi_{\mathrm{fin}}(\mathbb{R}^2)$ is a purely arithmetical
question, and therefore independent of set-theoretical subtleties.

\thingy If, however, we insist in working with $\chi(E^2)$ (and not
$\chi_{\mathrm{fin}}$) in the absence of choice, the results
formulated above are still applicable over certain fields.
Specifically, the facts that $\chi(\mathbb{Q}^2) = 2$
(par. \ref{remark-reduction-mod-4}), that
$\chi(\mathbb{Q}(\sqrt{2})^2) = 2$
(prop. \ref{chromatic-number-of-q-sqrt-2}), that
$\chi(\mathbb{Q}(\sqrt{3})^2) = 3$
(prop. \ref{chromatic-number-of-q-sqrt-3}), that
$\chi(\mathbb{Q}(\sqrt{7})^2) = 3$
(prop. \ref{chromatic-number-of-q-sqrt-7}), and that $4 \leq
\chi(\mathbb{Q}(\sqrt{3}, \sqrt{11})^2) \leq 5$
(prop. \ref{chromatic-number-of-q-sqrt-3-and-11}) still hold in the
absence of choice: the reason for this is that any choice which
requires the axiom in \ref{main-result-valuation-rings} can in fact be
done systematically for the specific fields considered here.  For
example, it \emph{does not} require the axiom of choice to select a
representative from each class of $\mathbb{Q}_3(\sqrt{3}) /
\mathbb{Z}_3[\sqrt{3}]$: one can simply write an element of
$\mathbb{Q}_3(\sqrt{3})$ in the form $\sum_{i=-N}^{+\infty} a_i
\sqrt{3}^{i}$ with $a_i \in \{0,1,2\}$ and choose the representative
$\sum_{i=-N}^{-1} a_i \sqrt{3}^{i}$.

\end{document}